\newcommand{\comment}[1]{}
\newcount\hh
\newcount\mm
\mm=\time
\hh=\time
\divide\hh by 60
\divide\mm by 60
\multiply\mm by 60
\mm=-\mm
\advance\mm by \time
\def\hhmm{\number\hh:\ifnum\mm<10{}0\fi\number\mm}
\newlength\templ

\documentclass{article}

\usepackage{amssymb, amsthm,amsmath}
\usepackage{pictexwd, dcpic}
\usepackage[]{hyperref}
\hypersetup{
colorlinks=true,
	linkcolor=black,
}
\usepackage{a4wide} 
\usepackage{fancyhdr}

\theoremstyle{definition} 
\newtheorem{theorem}{Theorem}
\newtheorem{prop}[theorem]{Proposition}
\newtheorem{lemma}[theorem]{Lemma}
\newtheorem*{defi}{Definition}

\newtheorem*{ack}{Acknowledgement}
\newtheorem{cor}[theorem]{Corollary}

\theoremstyle{remark} 
\newtheorem*{nota}{Notation}

\newcommand\ann{\mathrm{Ann}}

\newcommand\sep{\,:\,}
\newcommand\st {\mathrel{\ooalign{$\,\backepsilon$\cr\lower .7pt\hbox{\kern 1pt$-\,$}}}}
\newcommand\qf[1]{\mathrm{Quot}(#1)}	
\newcommand\spec{\mathrm{Spec}\,}	
\newcommand\sper{\mathrm{Sper}\,}	
\newcommand\supp{\mathrm{supp}}	
\newcommand\id{\mathrm{id}}	

\newcommand\minspec{\mathrm{MinSpec}\,}

\newcommand{\N}{\mathbb{N}}
\newcommand{\p}{\mathfrak{p}}
\newcommand\wo[1]{\backslash{\{#1\}}}
\newcommand{\bij}{\rightarrowtail \hspace{-1.9ex} \rightarrow}

\newcommand*{\lhrarrow}{\ensuremath{\lhook\joinrel\relbar\joinrel\rightarrow}}
\newcommand*{\thrarrow}{\twoheadrightarrow}
\newcommand*{\lthrarrow}{\ensuremath{\relbar\joinrel\twoheadrightarrow}}
\renewcommand\mod{\,\mathrm{mod}\,}
\newcommand{\ic}{\mathrm{ic}}

\def\blacksq{\begin{flushright}$\blacksquare$\end{flushright}}

\def\tboxit#1#2{{\setbox0=\hbox{\kern5pt#1\kern5pt}\edef\titlewidth{\the\wd0}%
                \setbox2=\vbox{#2} 
                \setbox2=\vbox{%
                    \vbox to1pt{\vss\hbox 
to\wd2{\strut\hfil#1\hfil}\vskip0pt}%
                    \box2 
                    } 
                \Tboxit{\titlewidth}{\Tboxit{\titlewidth}{\box2}}}} 

\def\Tboxit#1#2{\vbox{%
    \setbox0=\hbox{\vrule\kern3pt\vbox{\kern3pt#2\kern3pt}\kern3pt\vrule}%
    \hbox to\wd0{\hrulefill\kern#1\hrulefill}\nointerlineskip 
    \box0 
    \hrule 
    }}

\comment{
\makeatletter
\makeatother
}

\begin{document}

\title{Real closed $*$ reduced partially ordered Rings}
\makeatletter
\let\mytitle\@title
\makeatother
\author{Jose Capco \\
Email: \href{mailto:jcapco@yahoo.com}{\small{jcapco@yahoo.com}} 
}
\date{}
\thispagestyle{empty}
\maketitle

\pagestyle{fancy}
\fancyhead[R]{\mytitle}
\fancyhead[L]{J. Capco}
\tolerance=500

\begin{abstract}
In this work we attempt to generalize our result in \cite{Capco3} \cite{Capco4} for real rings (not just von Neumann regular real rings). 
In other words
we attempt to characterize and construct real closure $*$ of commutative unitary rings that are real. 
We also make some very interesting and significant discoveries regarding maximal partial orderings of rings, Baer rings and essentail 
extension of rings.
The first Theorem itself gives us a noteworthy bijection between maximal partial orderings of two rings by which one is a 
rational extension of the other. We characterize conditions when a Baer reduced ring can be integrally closed in its
total quotient ring. We prove that Baer hulls of rings have exactly one automorphism (the identity) and we even prove
this for a general case (Lemma \ref{deck_transform_Baer}). Proposition \ref{Prop_essext} allows us to study essential
extensions of rings and their relation with minimal prime spectrum of the lower ring. And Theorem 
\ref{adjoin_idemp} gives us a construction of the real spectrum of a ring generated by adjoining idempotents to a 
reduced commutative subring (for instance the construction of Baer hull of reduced commutative rings).

From most of the above interesting theories we prove that there is a bijection between the real closure $*$ 
of real rings (upto isomorphisms) and their maximal partial orderings.
We then attempt to develop some topological theories for the set of real closure $*$ of real rings (upto isomorphism)
that will help us give a topological characterization in terms of the real and prime spectra of these rings. The topological
characterization will be revealed in a later work. It is noteworthy to point out that we can allow ourself to 
consider mostly the minimal prime spectrum of the real ring in order to develop our topological theories.

\begin{description}
\item[Mathematics Subject Classification (2000):] Primary 13J25; Secondary 06E15, 16E50
\item[Keywords:] real closed $*$ rings, regular rings, absolutes of Hausdorff spaces, irreducible surjections, $f$-ring
partial orderings,total quotient ring, maximal partial ordering of rings, essential extensions of rings.
\end{description}
\end{abstract}

\footnote{Supported by Universit\"at Passau, Passau, Germany and Magna-Steyr, St. Valentin, Austria}

\begin{nota}
If $A$ is a commutative unitary ring, then we write $T(A)$ to mean the total quotient ring of $A$. And if
$A$ is partially ordered, with a partial ordering $A^+$, 
then we automatically assume a default partial ordering of $T(A)$, and unless otherwise defined we 
write $T(A)^+$ for it, which is the weakest partial ordering of $T(A)$ that extends $A^+$. In other words
$$T(A)^+ :=\{\sum_{i=1}^n a_it_i^2 : n\in \N,a_i\in A^+, t_i\in T(A), i=1,\dots,n\}$$
\end{nota}

\begin{theorem} \label{maxpo} Let $A$ be a subring of a reduced commutative ring $B$, suppose also that $B$ is a rational extension of $A$.
\item[(i)] Suppose that $B$ has a maximal partial ordering $B^+$, then the partial ordering $B^+\cap A$ is also a maximal
partial ordering of $A$.

\item[(ii)] There is a bijection $\Phi : \mathcal P_B \bij \mathcal P_A$, where $\mathcal P_B$ is the set of all maximal 
partial orderings of $B$ and $\mathcal P_A$ is the set of all maximal partial orderings of $A$. If we have a fixed partial ordering
of $A$, say $A^+$, and if $B^+$ is the weakest partial ordering of $B$ extending $A^+$ i.e.
$$B^+ :=\{\sum_{i=0}^n a_ib_i^2 \sep a_i\in A^+, b_i\in B\} $$
then we can similarly prove that there is a bijection between the set of maximal partial orderings of $B$ containing $B^+$ 
and the set of maximal partial ordering of $A$ containing $A^+$.
\end{theorem}
\begin{proof} 
\item[(i)]
Set $A^+:=B^+\cap A$, if $A^+$ were not a maximal partial ordering of $A$ then there exists an element $a\in A\backslash A^+$
that extends $A^+$ to another partial ordering of $A$, this is equivalent to (one can easily prove this or find this in
, \cite{Brum} Proposition 1.5.1)
\begin{equation*}
a_1a + a_2=0  \quad \Leftrightarrow\quad aa_1,a_2 =0 \qquad (a_1,a_2\in A^+)
\tag{*}\end{equation*}

Now $a\not\in B^+$ and because $B^+$ is a maximal partial ordering of $B$, $a$ cannot extend $B^+$ as a partial ordering of $B$ i.e 
$$\exists b_1,b_2 \in B^+\wo{0} \st  ab_1 + b_2 = 0$$

By \cite{Brum} \S1.4 p.38 and the definition of $A^+$, there exists an $a_2\in A^2\wo{0}$ such that $a_2b_2\in A^+\wo{0}$. Then we have the 
following cases

\vspace{5mm}\underline{Case 1: $a_2ab_1 = 0$}. Then 
$$a_2(ab_1+b_2) = a_2b_2 \neq 0$$
but $ab_1 + b_2 = 0$, here we have a contradiction!

\vspace{5mm}\underline{Case 2: $a_2ab_1 \neq 0$}. Then there is an $a_1\in A^2$ such that 
$$a_1a_2ab_1, a_1a_2b_1 \in A\wo{0}$$  
this also shows that $a_1a_2b_1\in A^+\wo{0}$. We thus have the following 
$$ a_1a_2(ab_1+b_2)=(a_1a_2b_1)a+(a_1a_2b_2)=0$$
but $a_1a_2b_1, a_2a_2b_2\in A^+\wo{0}$ and by ($*$) $a_1a_2ab_1 =0$ which is a contradiction. 

\vspace{5mm}
\item[(ii)] For $\tilde P\in \mathcal P_B$ define $\Phi(\tilde P) := \tilde P \cap A$. By (ii), 
$\Phi(\tilde P)\in \mathcal P_A$, and we need only show that $\Phi$ is both surjective and injective. We prove by
contradiction, assume there are $\tilde P_1,\tilde P_2 \in \mathcal P_B$ such that 
$$\Phi(\tilde P_1)=\Phi(\tilde P_2) =: P \in \mathcal P_A \quad \textrm{ for some } P\in \mathcal P_A$$
If $\tilde P_1\neq \tilde P_2$ then there exists a $b\in \tilde P_2\backslash \tilde P_1$. Because $\tilde P_1$ is a maximal
partial ordering of $B$, we have (see for instance \cite{Brum} Proposition 1.5.1) some $b_1,b_2\in \tilde P_1\wo{0}$ such that 
$bb_1 + b_2 =0$. Without loss of
generality we assume also that $b_1\in B^2\subset \tilde P_1,\tilde P_2$, since we can always write
$b_1^2b + b_2b_1 = 0$ knowing that (our rings are reduced) $b_1^2, b_2b_1 \in \tilde P_1\wo{0}$. So there is an 
$a_2 \in A^2 \subset P$ such that $a_2b_2\in P\wo{0}$  (by the definition of $P$). We have the following cases

\vspace{5mm}\underline{Case 1: $a_2b_1b = 0$}.  Then 
$$a_2(b_1b + b_2) = a_2b_2 = 0$$
a contradiction. 

\vspace{5mm}\underline{Case 2: $a_2b_1b \neq 0$}. Then, there is an $a_1\in A^2 \subset P$ such that 
$a_1a_2b_1b \in P\wo{0}$.
But we also have 
$$a_1a_2(b_1b+b_2) = \underbrace{a_1a_2b_1b}_{\in P\wo{0}} + \underbrace{a_1a_2b_2}_{\in P\wo{0}} = 0$$
This is a contradiction, as $P$ is a partial ordering of $A$ (see for instance \cite{Brum} Proposition 1.2.1(b)).

\vspace{5mm} Thus we have shown that $\Phi$ is injective. Now to show that $\Phi$ is surjective, consider any 
$P\in \mathcal P_A$. Consider $\tilde P$ to be a partial ordering of $B$ that is maximal and contains the partial ordering
of $B$ defined by 
$$\{\sum_{i=1}^n b_i^2a_i \sep a_i \in P, b_i\in B, i=1,\dots,n\}$$
(for the case $A$ has a given partial ordering $A^+$ and $P$ contains this $A^+$. We see that $\tilde P$ contains 
the weakest partial ordering of $B$ extending $A^+$). We observe then that 
$$\Phi(\tilde P) =\tilde P \cap A \supset P$$ 
but $P$ being a maximal partial ordering of $A$ implies then that $\Phi(\tilde P) = P$ and so we have shown that $\Phi$ is
surjective.
\end{proof}

The Theorem above just enhanced \cite{Capco4} Theorem 20 and so we can write

\begin{theorem}\label{vNr_maxpo}
Let $A$ be a real, regular ring then there exists a bijection between the following sets
\begin{enumerate}
\item $\{C \sep C$ is a real closure $*$ of $A\}/\sim$ \\ where 
for any two real closure $*$ of $A$, $C_1$ and $C_2$, one defines $C_1\sim C_2$ iff there is an
$A$-poring-isomorphism  between $C_1$ and $C_2$
\item $\{X\subset \sper A \sep X \textrm{ is closed and } \supp_A|X: X\rightarrow \spec A 
\textrm{ is an irreducible surjection}\}$
\item $\{P\subset A \sep P\supset A^+ $ and $P$ is a maximal partial ordering of $A\}$
\item $\{C \sep C$ is a real closure $*$ of $B(A)\}/\sim$ \\ where 
for any two real closure $*$ of $B(A)$, $C_1$ and $C_2$, one defines $C_1\sim C_2$ iff there is a 
$B(A)$-poring-isomorphism between $C_1$ and $C_2$
\item $\{s:\spec B(A) \rightarrow \sper B(A) \sep s$ is a continuous section of $\supp_{B(A)} \}$
\item $\{X\subset \sper B(A) \sep X$ is closed and $\supp_{B(A)}|X: X\rightarrow \spec B(A)$ 
is an irreducible \mbox{surjection$\}$}
\item $\{P\subset B(A) \sep P\supset B(A)^+ $ and $P$ is an $f$-ring partial ordering of $B(A)\}$
\end{enumerate}
\end{theorem}

We make the following Lemma, whose proof is quite straightforward, therefore it is omitted.

\begin{lemma}\label{f-ring-supsninf} Let $A$ be an $f$-ring then for any $x,y\in A$ the following identity holds
$$y + (x-y)^+  = x + (y-x)^+  = x \vee y$$ 
$$y - (x-y)^-  = x - (y-x)^-  = x \wedge y$$
\end{lemma}

\begin{theorem} \label{rcrs_totq}
\begin{enumerate}
\item[(i)] Let $A$ be a reduced ring integrally closed in a real closed von Neumann regular ring $B$. Let $f\in A[T]$ and 
$g\in B[T]$ be monic polynomials of odd degree (i.e. $\deg(f),\deg(g)\in 2\N +1$). Then $f$ has a zero in $A$ and $g$ has a zero in $B$.
\item[(ii)] Assume the rings $A$ and $B$ as above. Then $A$ has the property that $\qf{A/(\p\cap A)}$ is algebraically
closed in $B/\p$ for all $\p\in \spec B$ (i.e. $\qf{A/(\p\cap A)}$ is a real closed field).
\item[(iii)] Let $A$ be a reduced commutative unitary ring and $T(A)$ be von Neumann regular, then 
$$\qf{A/(\p\cap A)}=T(A)/\p \qquad \forall\p\in\spec T(A)$$
Now if this $A$ is as in (ii) and $T(A)$ an intermediate ring of $A$ and $B$ then 
$T(A)$ is in fact a real closed ring (not necessarily real closed $*$).
\end{enumerate}
\end{theorem}
\begin{proof} 
(i) First we show that $g$ has a zero in $B$. Let $\p \in \spec B$ then we know by the very definition
of real closed rings  that $B/\p$ is a real closed field. Thus the canonical image of $g$ in $B/\p[T]$, denote by 
$\widehat g$, has a zero in $B/\p$ (note that $g$ is monic and thus $\widehat g$ will have the same degree
as $g$) say $\widehat b_\p$, where $\widehat b_\p$ is the canonical image of some $b_\p\in B$
in $B/\p$. We can do this for any prime ideal $\p\in \spec B$. Now set 
$$V_\p := \{\mathfrak q  \in \spec B \sep f(b_\p) \in \mathfrak q \}$$
now because $B$ is von Neumann regular we know that $V_\p$ is a clopen set in $\spec B$, furthermore we know that 
$\p \in V_\p$. Thus 
$$\spec B = \bigcup_{\p\in \spec B} V_\p$$
Now because $\spec B$ is compact, there are $V_1,\dots, V_n \subset \spec B$ that are clopen and together they 
cover $\spec B$ and such that for any $i\in \{1,\dots, n\}$ we associate a $b_i\in B$ such that 
$$V_i = \{\p  \in \spec B \sep f(b_i) \in \p\}$$ 

We look at the global section ring of the sheaf structure of $B$ and we define mutually disjoint clopen sets $U_1, \dots, U_n$ by 
$$U_1 := V_1, \dots, U_i := V_i\backslash U_{i-1} \textrm{ for } i=2,\dots, n$$ 
Now define $b$ in the global section ring (which is actually isomorphic to $B$) by 
$$b(\p):= b_i(\p) \textrm{ if } \p \in U_i$$
Then we observe that for any $\p\in \spec B$ one has $g(b)\in \p$, and because $B$ is reduced we conclude that 
$g(b) = 0$. Thus $g$ has a zero in $B$.

Now because $f$ is monic and of odd degree, it has a zero in $B$ and because $A$ is integrally closed in $B$, this zero
must actually be in $A$.

\vspace{5mm}
\item[(ii)] Let $\p$ be in $\sper B$.
Set $K:=\qf{A/(\p\cap A)}$ and $L:=B/\p$, also for any $x\in B$ and $f\in B[T]$ denote $\widehat x$ and $\widehat f$ 
to be the canonical image of $x$ in $L$ and the canonical image of $f$ in $L[T]$ respectively.

Suppose now that $\widehat f$ is in $K[T]$, monic and of odd degree for some $f\in B[T]$. We may write 
$$\widehat f(T) = \sum_{i=0}^{n-1}\frac{\widehat{a_i}}{\widehat{a_n}}T^i + T^n \quad a_n\in A\backslash\p, a_i\in A, n\in 2\N+1$$
Define now $g(T)\in A[T]$ by 
$$g(T) := T^n + \sum_{i=0}^{n-1} a_i a_n^{n-i-1} T^i$$
Then $g$ is a monic polynomial of odd degree in $A[T]$ and by (i) we can conclude that $g$ has a zero,
say $a$, in $A$. Thus we may as well conclude that $\widehat a$ is a zero of $\widehat g$. 

Now we observe that 
$$ \widehat{a_n}^n \widehat{f}(T) = \widehat{g}(\widehat a_n T) $$
and because $\widehat a_n$ has an inverse in $K$ we learn that $\widehat{a}\widehat{a_n}^{-1}$ is a zero of $\widehat f$.
But $\widehat{a}\widehat{a_n}^{-1}$ is in $K$. Thus we have shown that any monic polynomial of odd degree in $K[T]$ has a 
zero in $K$.

We do know that $B$ is a real closed 
ring thus its partial ordering is (see for instance \cite{SM} Proposition 12.4(c))
$$ B^+ = \{b^2 \sep b\in B\} $$
Since $A$ is integrally closed in $B$, we conclude that the set
$$A^+ := B^+ \cap A = \{a^2 \sep a \in A\}$$ 
is a partial ordering of $A$. We show that $A$ with this partial ordering is actually a sub-$f$-ring of $B$.
By Lemma \ref{f-ring-supsninf}, we need only show that for any $a\in A$, $a^+\in B$ is in $A$ (and thus so is $a^-$).  $B$ is a 
von Neumann regular ring, so $a^+$ has a quasi-inverse we shall denote by $(a^+)'$. Since $A$ is integrally closed in $B$, 
we then also know that the idempotent $a^+(a^+)'$ is in $A$, but then 
$$a(a^+(a^+)') = (a^+ - a^-)(a^+(a^+)')=a^+(a^+(a^+)')=(a^+)^2(a^+)' = a^+\in A$$ 

We will now show that $K^+$ defined by 
$$K^+ := \{(\widehat a/\widehat b)^2 \sep a\in A, b\in A\backslash\p \}$$
is a total ordering of $K$ (and thus by \cite{KS} p.16 Satz 1, $K$ is a real closed field). Since $A$ is a sub-$f$-ring 
of $B$ (the partial ordering of both is their weakest partial ordering) and $\p\cap A$ is a prime $l$-ideal of $A$ 
(we know that $B$ is a real closed regular ring, and by \cite{Capco} Proposition 7, all of it's residue fields are
real closed ring and so by \cite{BKW} Corollaire 9.2.5, $\p$ is an $l$-ideal of $B$. It is then easy to see
that it's restriction to $A$ is also an $l$-ideal), 
we know by \cite{BKW} Corollaire 9.2.5, that $A/(\p\cap A)$ is totally ordered by 
$$A^+/(\p\cap A) = \{\widehat a^2 \sep a\in A\}$$
One easily checks that this total ordering of $A/(\p\cap A)$ induces a total ordering of $K$ which is non other than
$K^+$.

\item[(iii)] Let $\p$ be in $T(A)$, clearly 
$$A/(\p\cap A) \hookrightarrow T(A)/\p$$
we show that all elements of $T(A)/\p$ is actually an element in $\qf{A/(\p\cap A)}$. For any $q\in T(A)$ we denote
$\bar q$ as the image of $q$ in $T(A)/\p$. 

Let $q\in T(A)\backslash \p$, then $\bar q$ is non-zero in $T(A)/\p$. There is a regular element $a\in A$ such that
$aq \in A$, and because $a$ is regular it cannot be contained in $\p$, otherwise its contained in $\p\cap A$ which is
a minimal ideal in $A$ (see \cite{Mewborn2} Theorem 3.1 and Theorem 4.4. Note that $\spec Q(A)\rightarrow \spec T(A)$ is a surjection because 
$T(A)$ is a regular ring, see for instance \cite{raphael} Lemma 1.14), but all regular elements of $A$ are not in any minimal ideal of $A$). Thus
since $\p$ is prime, we learn that $aq \in A\backslash \p$. Thus $\bar a, \bar a\bar q \in (A/(\p\cap A))^*$, thus
$\bar a$ has an inverse $\bar a^{-1}$ in $\qf{A/(\p\cap A)}$ and so 
$$\bar a^{-1}\bar a\bar q = \bar q \in \qf{A/(\p\cap A)}$$

So $T(A)/\p$ is a subring of $\qf{A/(\p\cap A)}$. Now $T(A)$ itself is a regular ring, 
so $T(A)/\p$ must be a field and because $A$ is a subring of $T(A)$ we get
$\qf{A/(\p\cap A)}$ as a subring of $T(A)/\p$. Therefore 
$$T(A)/\p = \qf{A/(\p\cap A)}$$
Because $T(A)$ is a subring of $B$ and both are regular rings, we then know that any prime ideal of $T(A)$ is a restriction
of prime ideal of $B$ (see \cite{raphael} Lemma 1.14). Thus by (ii) we can conclude that $T(A)/\p$ is a real closed field for any prime ideal 
$\p \in T(A)$. By \cite{Capco} Proposition 7, $T(A)$ must be a real closed ring.
\end{proof}

\begin{lemma}\label{Baertqr=vNr}
If $A$ is a Baer reduced commutative unitary ring, then $T(A)$ is von Neumann regular.
\end{lemma}
\begin{proof}
Let $a,b\in A$ and consider the ideal $I=aA+bB$ then because $A$ is Baer, there
is an idempotent $e\in E(A)$ such that $\ann_A(I)=eA=\ann_A(1-e)$. 
By \cite{Mewborn} Proposition 2.3 and then \cite{huckaba} Theorem B, we conclude that $T(A)$ is a regular ring. 
\end{proof}

\begin{theorem} \label{TqrAndIc}
Let $A$ be a Baer ring, then $A$ is integrally closed in $T(A)$ iff for any $\p\in\spec T(A)$ we have $A/(\p\cap A)$ is integrally
closed in $T(A)/\p$.
\end{theorem}
\begin{proof}
For simplicity let us set $B:=T(A)$.

"$\Rightarrow$" Suppose by contradiction there exists an $f\in A[T]$ monic, $\p\in\spec B$ and $b\in B$ such that
\begin{itemize}
\item[i.] $f(b)\in\p$ 
\item[ii.] $(b+\p)\cap A = \emptyset$
\end{itemize}

Then we have the following cases $\dots$ 

\vspace{5mm}\noindent \underline{Case 1: $b\in \p$}. Then $b\equiv 0\mod\p$ and this is a contradiction to ii.

\vspace{5mm}\noindent \underline{Case 1: $b\not\in \p$}. Then $f(b)\in \p\subset B$. Now because $A$ is Baer we know 
by Lemma \ref{Baertqr=vNr} that 
$B$ is von Neumann regular. Let $c\in B$ be the quasi-inverse 
of $f(b)$ . Then $cf(b)$ is an 
idempotent in $B$ and so it must be in $A$ (because $A$ is integrally closed in $B$). Now $1-cf(b)$ is also an idempotent, denote $e:=1-cf(b)$ 
(clearly $e\not\in\p$ because $1-e\in\p$) and observe that $ef(b)=0$. Now we can write 
$$ f(T) = T^n + \sum_{i=0}^{n-1} a_iT^i$$
for some $n\in\N$ and $a_0,\dots,a_n\in A$. Then 
$e^n f(b)=0$ and so $eb$ is a zero of the monic polynomial $g\in A[T]$ defined by
$$g(T) = T^n + \sum_{i=0}^{n-1} a_ie^{n-i}T^i$$
but because $A$ is integrally closed in $B$, we then know that $eb\in A$. 

But all these implies that $eb\equiv b\mod\p$ (since $1-e\in \p$) and so $(b+\p)\cap A \neq \emptyset$. Again a contradiction!

"$\Leftarrow$" Let $f\in A[T]$ be monic and $f(b)=0$ for some $b\in B$. For any $\p\in \spec B$, consider $a_\p\in A$ to be such that
$f(a_\p)\in\p$ and $a_\p\equiv b\mod\p$. Consider the clopen sets (because $A$ is Baer and so $B$ is regular):
$$ V_\p := \{\mathfrak q\in \spec B : a_\p \equiv b\mod \mathfrak q\}$$
then $\p\in V_\p$ for all $\p\in\spec B$ 
$$\bigcup_{\p\in\spec B} V_\p = \spec B$$
Thus there are finitely many
$a_1,...,a_n\in A$ such that 
$$\bigcup_{i=1}^n V_i = \spec B$$
where 
$$V_i = \{\mathfrak q\in \spec B : a_i \equiv b\mod \mathfrak q\}$$
Define now another family of clopen set 
$$U_1 := V_1, U_i = V_i\backslash U_{i-1} i\geq 2$$
we can also define the idempotents $e_i\in B$ by
$$e_i \mod \p =\left\{
\begin{array}{ll}
1  & \p \in U_i\\
0	& \p\not\in U_i\\
\end{array}
\right.$$

Clearly $b= \sum_{i=1}^n a_ie_i$
and all $a_i,e_i\in A$ for $i=1,\dots, n$ ($E(A)=E(T(A))$ because $A$ is Baer, see for instance \cite{raphael} 
). Thus $b\in A$!
\end{proof}

\begin{cor}
A reduced Baer poring $B$ is real closed $*$ iff for any minimal prime ideal $\p\in \minspec B$ one has 
$B/\p$ is a real closed $*$ integral domain.
\end{cor}
\begin{proof}
"$\Rightarrow$" If $B$ is real closed real $*$, then it is Baer and it is integrally closed in $Q(B)$, and furthermore 
$Q(B)$ is real closed $*$. Thus by Theorem \ref{rcrs_totq} $T(B)$ is also real closed $*$ (as it is also Baer and by 
Theorem \ref{rcrs_totq} iii a real closed regular ring, 
we can then use \cite{Capco} Theorem 15). Using \cite{SV} Proposition 2 and by Theorem \ref{TqrAndIc} we know 
that $B/\p$ is a real closed ring $*$ 
for any minimal prime ideal $\p$ in $\spec B$ (this is because the restriction of prime ideals of $T(B)$ to $B$
are exactly the minimal prime ideals of $B$, see for instance \cite{Mewborn2} Theorem 3.1 and Theorem 4.4),

"$\Leftarrow$" The same reasons as above shows us that $B$ is integrally closed in $T(B)$ and $T(B)$ is a real 
closed $*$ ring. Now $Q(B)$ is also the complete ring of quotients of $T(B)$ so by \cite{srcr} Theorem 3, $Q(B)$ is also 
a real closed ring $*$ and in this case $T(B)$ is obviously integrally
closed in $Q(B)$. Our initial hypothesis and results in Theorem \ref{TqrAndIc}, \cite{SV} Proposition 2, \cite{Mewborn2} Theorem 3.1 and Theorem 4.4,
does imply that $B$ is integrally closed in $Q(B)$. All this implies satisfies the condition of \cite{srcr} Theorem 3 for B, making
us conclude that $B$ is real closed $*$.
\end{proof}

\noindent From the proof of the Corollary above we also immediately have the following

\begin{lemma} \label{rcrsIffTqr}
A poring $A$, is real closed $*$ iff it is integrally closed in its total quotient ring and its total
quotient ring is a real closed regular (Baer) ring.
\end{lemma}


\noindent Now we show one way how a real closure $*$ of a reduced ring can be found.

\begin{cor}\label{rcrs_poring_Q} If $A$ is a reduced poring and $B$ is a rationally complete real closed ring (thus also real 
closed $*$, see \cite{Capco} Theorem 15) such that 
$A$ is a sub-poring of it and $B$ is an essential extension of $A$. Then 
$\ic(A,B)$ is a real closure $*$ of $A$.
\end{cor}
\begin{proof}
Denote $\bar A := \ic(A,B)$. By Storrer's Satz one has the following commutative diagrams 
$$
\begindc{\commdiag}[1]
\obj(0,50)[A]{$A$}
\obj(80,50)[B]{$B$}
\obj(0,0)[C]{$\bar A$}
\obj(80,0)[D]{$Q(\bar A)$}
\mor{A}{B}{}[1,3]
\mor{D}{B}{}[1,3]
\mor{A}{C}{}[1,3]
\mor{C}{D}{}[1,3]
\enddc
$$
where all mappings above are canonical and the mapping $Q(\bar A) \rightarrow B$ is a monomorphism (of porings) 
as a result of Storrer's Satz. We therefore regard every poring in the commutative diagram above as a sub-poring
of $B$. Because $\bar A$ is integrally closed in $B$ it is integrally closed in $Q(\bar A)$, thus $\bar A$ is Baer (see
\cite{Mewborn} Proposition 2.5). It remains, 
by \cite{srcr} Theorem 3 (which we also partially needed to confirm, see \cite{Capco}),
to show that $Q(\bar A)$ is actually a real closed $*$ von Neumann regular ring. 

Now since $T(\bar A)$ is an intermediate ring of $Q(\bar A)$ and $\bar A$, 
it must also be an intermediate ring of $\bar A$ and $\bar B$. $\bar A$ is integrally closed in $B$ and so
by Theorem \ref{rcrs_totq} (iii) we conclude that $T(\bar A)$ is a real closed ring, and because $\bar A$ is Baer
we can also conclude that $T(\bar A)$ is a real closed $*$ ring (see \cite{Capco} Theorem 15). The complete ring of quotients of $T(\bar A)$ is 
also $Q(\bar A)$ and by \cite{srcr} Theorem 3 we know then that $Q(\bar A)$ must be a real closed $*$ ring.
\end{proof}

\begin{prop}\label{srcr_Toq1}
Let $B$ be a Baer reduced poring, then there is a bijection between the following sets
\begin{enumerate}
\item $\mathcal S := \{C\sep C$ is a real closure $*$ of $B\}/\sim$ where $C_1\sim C_2$ iff $C_1\cong_A C_2$
\item $\mathcal T := \{C'\sep C'$ is a real closure $*$ of $T(B)\}/\wr$ where $C_1'\wr C_2'$ iff $C_1'\cong_{T(A)} C_2'$
\end{enumerate}
\end{prop}
\begin{proof}
Let us define a map $\Phi: \mathcal S \rightarrow \mathcal T$. So let $C$ be a real closure $*$ of $B$ then we 
have the following canonical morphism of poring
$$B\lhrarrow C \lhrarrow T(C)$$
By Lemma \ref{rcrsIffTqr}, $T(C)$ is a real closed $*$ ring. Now $T(A)$ is actually a sub-poring of $T(C)$ because all regular element
of $A$ hava a (unique) inverse in $T(C)$. Moreover $T(A)$ is a Baer von Neumann regular ring (by Lemma \ref{Baertqr=vNr}).
We now define
$$C' := \ic(T(A),T(C))$$
This is by (\cite{Capco2} Proposition 6) a real closure $*$ of $T(A)$. So we finally can define $\Phi(C/\sim) = C'/\wr$.
We claim $\dots$

\vspace{5mm}\noindent
\underline{Claim 1: $\Phi$ is well-defined}. If $C_1\cong_A C_2$ then we have the following commutative diagram in the category of 
porings
$$
\begindc{\commdiag}[1]
\obj(0,40)[A]{$A$}
\obj(60,60)[C1]{$C_1$}
\obj(60,20)[C2]{$C_2$}
\obj(120,80)[T1]{$T(C_1)$}
\obj(120,0)[T2]{$T(C_2)$}
\obj(20,80)[TA1]{$T(A)$}{}
\obj(20,0)[TA2]{$T(A)$}{}
\mor{A}{C1}{}
\mor{A}{C2}{}
\mor{C1}{T1}{}
\mor{C2}{T2}{}
\mor{T1}{T2}{$f$}
\mor{A}{TA1}{}
\mor{A}{TA2}{}
\mor(20,80)(100,80){}
\mor(20,0)(100,0){}
\enddc
$$
where all the maps are canonical essential extensions, 
with $f:T(C_1)\rightarrow T(C_2)$ being the canonical extension of the $A$-isomorphism 
$$C_1\stackrel{\sim}{\longrightarrow}C_2$$
Thus $f|A$ is no other than the identity morphism from $A$ to $A$. And one easily checks that $f|T(A)$ is an identity map (as $T(A)$ is
no other than the localization of $A$ with respect to the multiplicative set consisting of the regular elements of $A$). Moreover if we 
define 
$$C_1' := \ic(T(A),T(C_1))$$
$$C_2' := \ic(T(A),T(C_2))$$
then we realize that $C_1'\cong_{T(A)}C_2'$. Thus we have shown that $C_1'\wr C_2'$ or in other words 
$$\Phi(C_1/\sim)=\Phi(C_2/\sim)$$

\vspace{5mm}\noindent
\underline{Claim 2: $\Phi$ is surjective}. Let $C'$ be a real closure $*$ of $T(A)$, then we have the following canonical injection
$$A\hookrightarrow T(A) \hookrightarrow C' \hookrightarrow Q(C')$$
by Corollary \ref{rcrs_poring_Q} we know that $\ic(A,Q(C'))$ is a real closure $*$ of $A$ and since $C'$ itself is a real closed $*$ ring, we have
$$\ic(A,C')=\ic(A,Q(C'))$$
we can then set $C:=\ic(A,C')$. The claim is that $\Phi(C/\sim)=C'/\wr$. 

Note that $T(A)$ is regular, so $C'$ must be regular 
(this can be seen for instance in \cite{Capco2} Proposition 7, or \cite{raphael} Lemma 1.9) and so we actually 
have $T(C')=C'$. Moreover we observe that 
$T(A)$ is a subring of $T(C)$ we thus have the following commutative diagram in the category of porings 
$$
\begindc{\commdiag}[1]
\obj(0,50)[A]{$A$}
\obj(60,50)[C]{$C$}
\obj(60,0)[TA]{$T(A)$}
\obj(120,50)[TC]{$T(C)$}
\obj(170,50)[C2]{$C'$}
\mor{A}{C}{}[1,3]
\mor{A}{TA}{}[1,3]
\mor{TA}{TC}{}[1,3]
\mor{C}{TC}{}[1,3]
\mor{TC}{C2}{}[1,3]
\enddc
$$
$C$ being a 
real closed $*$ ring implies 
(by Lemma \ref{rcrsIffTqr}) that $T(C)$ is a real closed $*$ ring. Thus we can conclude that 
$$\ic(T(A),T(C))=\ic(T(A),C')= C'=T(C)$$
This not only shows that $\Phi$ is surjective, but also the fact that  $\Phi(C/\sim)=T(C)/\wr$ for any real closure $*$, $C$, of $A$.

\vspace{5mm}\noindent
\underline{Claim 3: $\Phi$ is injective}. Let $C_1$ and $C_2$ be two real closure $*$ of $A$ so that 
(using the extra information we have learned in the previous proof) 
$$T(C_1) \cong_{T(A)} T(C_2)$$
We have immediately 
$$T(C_1)\cong_A T(C_2) \Rightarrow C_1\cong_A C_2 \Rightarrow C_1\sim C_2$$
\end{proof}

\begin{cor}\label{Baer_red_char1}
Let $B$ be a reduced Baer poring. Then there is a bijection between the following sets
\begin{enumerate}
\item $\mathcal S := \{C \sep C$ is a real closure $*$ of $B\}/\sim$ \\ where 
for any two real closure $*$ of $B$, $C_1$ and $C_2$, one defines $C_1\sim C_2$ iff there is an
$B$-poring-isomorphism  between $C_1$ and $C_2$
\item $\mathcal T := \{C' \sep C'$ is a real closure $*$ of $T(B)\}/\wr$ \\ where 
for any two real closure $*$ of $T(B)$, $C_1'$ and $C_2'$, one defines $C_1'\wr C_2'$ iff there is an
$T(B)$-poring-isomorphism  between $C_1'$ and $C_2'$
\item $\{P\subset B \sep P\supset B^+ $ and $P$ is a maximal partial ordering of $B\}$
\end{enumerate}
\end{cor}
\begin{proof}
"1$\Leftrightarrow$ 2" Proven in Propositon \ref{srcr_Toq1}

"2$\Leftrightarrow$ 3"  We know that there is a bijection between the real closure $*$ of a von Neumann 
regular ring and its maximal partial ordering, see Theorem \ref{vNr_maxpo}. Now the bijection between the set of maximal partial
ordering of $T(B)$ containing $T(B)^+$ and the set of maximal partial orderings of $B$ containing $B^+$ has been shown in
Theorem \ref{maxpo}(ii).
\end{proof}

\begin{defi}
If $A$ is a reduced commutative unitary ring, then we shall call any element of $Q(A)$ a \emph{rational element} or a
\emph{fraction} of $A$.
\end{defi}

\begin{lemma}\label{deck_transform_Baer} 1.) Let $A$ be a reduced ring and suppose that $C_1,C_2$ be reduced rings having $B(A)$ as subring. If 
$$f : C_1\stackrel{\sim}{\rightarrow} C_2$$
is an $A$-isomorphism then it is a $B(A)$-isomorphism.
\item[2.)] Let $A$ be a reduced poring and  $C$ be any real closure $*$ of $A$, one may then consider $B(A)$ as an intermediate
ring of $A$ and $C$ and $C$ itself is also a real closure $*$ of $B(A)$.
\end{lemma}
\begin{proof}
1.) For brevity we write $B:=B(A)$. First we show that $f(E(B))\subset E(B)$. We do know that $f(E(B))\subset E(C_2)$, 
now let $e_1\in E(B)\wo{0}$ and define $e_2:=f(e_1)\in E(C_2)$. We know that $e_1$ is a rational element of $A$, 
i.e. there exists an $a\in A$ such that $ae_1\in A\wo{0}$. Thus 
$$ae_1=f(ae_1)=af(e_1)=ae_2\in A\wo{0}$$
meaning that $e_2$ is also a rational element of $A$, in other words $e_2\in E(B)$ (see \cite{Mewborn} Proposition 2.5). 
Thus we have shown that $f(E(B))\subset E(B)$.

If $f|B$ were not the identity map then, because of 
 \cite{Mewborn} Proposition 2.5, there is an $e_1\in E(B)$ such that $f(e_1)\neq e_1$. Define 
$e_2:=f(e_1)\in E(B)$, then either $(1-e_1)e_2\neq 0$ or $(1-e_2)e_1\neq 0$. 
Suppose the former case holds
and define $e_3:= (1-e_1)e_2 \in E(B)$. Then there exists an 
$a\in A\wo{0}$ such that $ae_3, ae_2\in A\wo{0}$. So we get 
$f(a_3e_3)=a_3e_3\neq 0$ and yet 
$$f(a_3e_3)=f(a_3(1-e_1)e_2) = f(a_3e_2)f(1-e_1) = a_3e_2(1-e_2)=0$$
a contradiction. If the former case holds, then we get a contradiction in a similar manner.

\vspace{5mm}
\item[2.)] This has been discussed throughout our study of regular rings, but we shall give a formal prove for all reduced
rings here. The keyword is \emph{Storrer's Satz} (see \cite{Capco3} Theorem 1). By Storrer's Satz we have the following 
commutative diagram consisting of canonical maps (except for the lower horizontal map which is due to Storrer's Satz).
$$
\begindc{\commdiag}[1]
\obj(0,50)[A]{$A$}
\obj(80,50)[B]{$C$}
\obj(0,0)[C]{$Q(A)$}
\obj(80,0)[D]{$Q(C)$}
\mor{A}{B}{}[1,3]
\mor{B}{D}{}[1,3]
\mor{A}{C}{}[1,3]
\mor{C}{D}{}[1,3]
\enddc
$$
$C$ being real closed $*$ must contain all the idempotents of $Q(C)$, thus it contains all the idempotents of $Q(A)$ and has
$A$ as a subring. This implies (by \cite{Mewborn} Proposition 2.5) that $C$ has $B(A)$ as a subring. That $C$ is an integral
and essential extension of $B(A)$ is then clear. $C$ being real closed $*$, means that $C$ is a real closure $*$ of $B(A)$ 
as well.
\end{proof}

\begin{theorem}\label{redring-rcrs} Let $A$ be a reduced poring, then there is a bijection between the following sets
\begin{itemize}
\item $\mathcal A := \{C\sep C$ is a real closure $*$ of $A\}/\sim$ where $C_1\sim C_2$ iff $C_1\cong_A C_2$
\item $\mathcal B := \{C'\sep C'$ is a real closure $*$ of $B(A)\}/\wr$ where $C_1'\wr C_2'$ iff $C_1'\cong_{B(A)} C_2'$
\item $\mathcal P_{B(A)}:= \{P\subset B(A) \sep P\supset B(A)^+ $ and $P$ is a maximal partial ordering of $B(A)\}$
\item $\mathcal P_{A}:= \{P\subset A \sep P\supset A^+ $ and $P$ is a maximal partial ordering of $A\}$
\end{itemize}
\end{theorem}
\begin{proof}
In our proof, when we say \emph{isomorphism} we mean it in the category of porings.

\vspace{5mm}\noindent
"$\mathcal A \bij \mathcal B$" Any real closure $*$ of $A$ is also a real closure $*$ of $B(A)$ (see Lemma  \ref{deck_transform_Baer} (2)), this is
also true vice versa, as $B(A)$ itself is an essential and integral extension of $A$. If two such real closure $*$ are 
$A$-isomorphic, then they are by  Lemma \ref{deck_transform_Baer} (1) also $B(A)$-isomorphism. Any $B(A)$-isomorphism 
is also trivially an $A$-isomorphism. Thus the bijection is just the canonical map 
$$(C/\sim)\, \longmapsto \,(C/\wr)\qquad C\textrm{ is a real closure $*$ of } A$$

\vspace{5mm}\noindent
"$\mathcal B \bij \mathcal P_{B(A)}$" This is from Corollary \ref{Baer_red_char1}

\vspace{5mm}\noindent
"$\mathcal P_{B(A)} \bij \mathcal P_{(A)}$" This is from Theorem \ref{maxpo} (ii)
\end{proof}

\begin{prop}\label{Prop_essext} Let $B$ be a commutative unitary ring with a subring $A$ and let 
$$\phi:\spec B\rightarrow \spec A$$ be the canonical continuous map and define $X:=\phi(\spec B)$ with
relative topology, then $\dots$
\begin{itemize}
\item[i.] For any $a\in A$ one has 
$\phi(V_B(a))=V_A(a)\cap X$ and $\phi(D_B(a))=D_A(a)\cap X$
\item[ii.] For all $a\in A$ and $\p \in D_A(a)\cap X$ one has $\phi^{-1}(\p) \subset D_B(a)$ 
\item[iii.] If $B$ is reduced and an essential extension of $A$ then $\phi$ induces an irreducible surjection
$$\phi' : \spec B \twoheadrightarrow X$$
with $\phi'(\tilde \p) := \phi(\tilde \p)$ for all $\tilde \p\in \spec B$.
\item[iv.] If $B$ is reduced and an essential extension of $A$, and define $Y:=\minspec A$. Suppose now that 
$Y\subset X$ and set $\tilde Y:= \phi^{-1}(Y)$ then
$\phi$ can be restricted to a map 
$$\tilde Y \thrarrow Y \qquad \tilde\p\mapsto \phi(\tilde\p) \quad\forall\tilde\p\in\tilde Y$$
and this map is an irreducible surjection with relative topologies on the domain and range.
\item[v.] Let $B$, $Y$ and $\tilde Y$ be as in (iv). Then for any $a\in A$ we have the identity that 
$$\phi(D_B(a)\cap \tilde Y)=D_A(a) \cap Y$$ 
Moreover $\tilde Y$ is dense in $\spec B$. 
\end{itemize}
\end{prop}
\begin{proof} i. Let $\tilde\p\in \spec B$ and such that $a\in \tilde \p$ then clearly $\phi(\tilde \p)=\p\cap A$ has $a$ in it and so
$\phi(\tilde\p)\in V_A(a)\cap X$. The other containment is equally obvious. Analogously one proves the second equality.

\item[ii.] Suppose $a\in A$ and let $\tilde \p\in \spec B$ such that $\tilde \p\cap A=\p$ and that $\p\in D_A(a)$ then 
if $a\in \tilde \p$ we had get a contradiction since then $a\in \tilde \p \cap A$. 

\item[iii.]  Observe that for any $b\in B\wo{0}$, there is a $c\in B$ such that $bc\in A\wo{0}$ and that 
$D_B(b)\supset D_B(bc)$. So if we try to prove by contradiction, we may assume that there is an $a\in A\wo{0}$ such that
$\phi(V_B(a))=X$, in other words for a $\p\in D_A(a)\cap X$ there is a $\tilde \p\in V_B(a)$ such that 
$\tilde \p\cap A = \p$. But this cannot because in ii. we have shown that 
$$\tilde \p \in \phi^{-1}(\p) \subset D_B(a)$$
a contradiction!

\item[iv.]
\comment{
If $a\in A\wo{0}$ then $D_A(a)\cap Y \neq \emptyset$, since otherwise $V_A(a)\cap Y =Y$ implies that 
$a\in\p$ for all $\p\in\minspec A$ and because $A$ is reduced this implies that $a=0$ which is a contradiction.

\vspace{3mm}\noindent
Now if $a\in A\wo{0}$ and if $\p\in D_A(a)\cap Y$ we know then by the fact that $Y\subset X$ there is a 
$\tilde\p\in \tilde Y$ such that 
$$\phi(\tilde \p)=\tilde \p\cap A= \p$$
So we easily observe that $\tilde \p \in D_B(a)\cap \tilde Y$.

\vspace{3mm}\noindent
We can therefore write the following equivalence for all $a\in A$
$$D_B(a)\cap \tilde Y =\emptyset \Leftrightarrow  D_A(a)\cap Y =\emptyset \Leftrightarrow  D_A(a)\cap X =\emptyset
\stackrel{(ii)}{\Longleftrightarrow}  D_B(a) =\emptyset$$
The middle equivalence is due to the fact that $Y\subset X$ and that our ring is reduced.

\vspace{3mm}\noindent
}
We prove (iv) by contradiction, and 
we may assume that there is an $a\in A\wo{0}$ 
that defines $\tilde Y_1 :=V_B(a)\cap \tilde Y$ and such that $\phi(\tilde Y_1)=Y$. So we get (using i.)
$$\phi(V_B(a)) =V_A(a)\cap X \supset \phi(\tilde Y_1)=Y$$
Since $V_A(a)\cap X$ is closed in $X$ it should contain all the closure points of $Y$ in $X$, thus
$\phi(V_B(a))=X$ but this contradicts (iii).

\item[v.] Clearly (see also (i)) one has for any $a\in A$
$$\phi(D_B(a)\cap \tilde Y)\subset D_A(a) \cap Y$$
The equality of the above is clear for $a=0$ so we need only deal for the case $a\in A\wo{0}$. Let thus $a\in A\wo{0}$ and
suppose $\p \in D_A(a)\cap Y$. Since $Y\subset X$ we know that there is a $\tilde\p\in \tilde Y$ such that 
$$\phi(\tilde \p)=\tilde \p\cap A= \p$$
and we easily see also that $\tilde \p \in D_B(a)\cap \tilde Y$. Thus 
$\p \in \phi(D_B(a)\cap \tilde Y)$ and so we may conclude in the end that
$$\phi(D_B(a)\cap \tilde Y)\subset D_A(a) \cap Y$$

\vspace{3mm}\noindent
Suppose now that $\tilde Y$ is not dense in $\spec B$, this just implies that there is a $b\in B$ such that 
$$D_B(b) \cap \tilde Y = \emptyset$$
Now since $B$ is reduced and is essential over $A$ and since $B$ is reduced, there exists a $b'\in B$ such that
$b'b\in A\wo{0}$. Thus since $D_B(bb')\subset D_B(b)$ we may as well say that there is an $a\in A\wo{0}$ such that
$$D_B(a) \cap \tilde Y = \emptyset$$
But the previous results implies then that 
$$D_A(a) \cap Y=\emptyset $$
This implies that 
$$V_A(a) \cap Y=Y$$
meaning that for any $\p\in\minspec A$ we have $a\in \p$ or in other words (since $A$ is also reduced) we get $a=0$ 
and this is contradiction. Thus we may conclude that $\tilde Y$ is dense in $\spec B$.
\end{proof}

\begin{theorem} \label{adjoin_idemp}
Let $A$ be a reduced poring and $B$ be an over-ring of $A$. Define 
$$ C:=\{\sum_{i=1}^n a_ie_i \sep \forall n\in \N, a_1,\dots,a_n\in A, e_1,\dots,e_n\in E(B)\} $$
Then (in the category of topological spaces)
$$\sper C \cong \spec C \times_{\spec A} \sper A$$
\end{theorem}
\begin{proof}
The proof of this is done by several inductions. First we show the following claim (we shall end the proof of all
claims in the Theorem by a black square, and proof of the Theorem itself is ended by a white square):

\vspace{5mm}
\noindent\underline {Claim 1:} If $\alpha \in \sper A$ and $\tilde p \in \spec C$ be such that 
$\supp_A(\alpha) = \tilde p \cap A$ then there exists an $\tilde \alpha \in \sper C$ such that 
$\supp_C(\tilde \alpha)=\tilde \p$ and $\tilde \alpha \cap A=\alpha$.

\noindent\emph{Proof of Claim 1.} We consider several cases $\dots$ 

\vspace{5mm}\noindent\emph{\textbf{Case 1.1:}}  $C=A[e]$ for some $e\in E(B)$

\noindent Thus in this case we may every now and then regard $C$ as 
$$C\cong Ae\times A(1-e)$$
where $Ae$ and $A(1-e)$ are considered as commutative rings with $e$ respectively $1-e$ as 
their unity (with canonical multiplication and addition, as derived from the ring $A$). Consider now the set 
$$\tilde \alpha = \alpha[e] + \tilde p \subset C$$
we claim  that $\tilde \alpha$ is a prime cone in $C$. Clearly $\tilde \alpha$ is closed under multiplication and addition.
And all the squares in $C$ are in $\tilde \alpha$. Thus it suffices if we prove that 
$\tilde\alpha \cap -\tilde\alpha = \tilde p$ and $\tilde\alpha \cup -\tilde \alpha = C$. 
Clearly $\tilde \p\subset\tilde\alpha \cap -\tilde\alpha$ and $\tilde\alpha \cup -\tilde\alpha\subset C$
so we need only show that $\tilde\alpha \cap -\tilde\alpha \subset \tilde \p$ and 
$C\subset\tilde\alpha \cup -\tilde\alpha$. 
Suppose $c,d \in A$ and consider 
$$c(1-e)+de \in A[e] = C\cong Ae\times A(1-e)$$
also let $a,b,a',b'\in \alpha$ with $x,y\in \tilde p$ such that 
$$ae+b+x=-a'e-b'-y$$
i.e. $ae+b+x\in \tilde \alpha \cap \tilde -\alpha$.  We have two cases $\dots$ 

\vspace{5mm}\noindent\emph{\textbf{Case 1.1.1:}}  $e\in \tilde p$

\noindent Then $b\equiv -b'\mod\tilde \p$ which implies that 
$b+b'\in \tilde p \cap A=\supp_A(\alpha)$. From this and because $b,b'\in \alpha$ we get that $b,b'\in \supp_A(\alpha)\subset \tilde\p$.
Thus we obtain that $ae+b+x \in \tilde p$.

Also, if now $c\in \alpha$ we know that $c+(d-c)e\in \alpha + \tilde \p\subset \alpha[e]+\tilde\p$. And if $c\in -\alpha$ 
(we know that $c\in A=\alpha\cup-\alpha$) we get $c+(d-c)e\in -(\alpha[e]+\tilde \p)$.

\vspace{5mm}\noindent\emph{\textbf{Case 1.1.1:}} $1-e\in \tilde p$

\noindent We claim that $ae+b\in \tilde \p$. We write $(a+a')e + (b+b') = -x-y \in \tilde \p$. Let's set $z=-(x+y)$, we note
that we can write $z$, uniquely, as 
$$z=z_1e + z_2(1-e) \qquad z_1,z_2\in A$$
Because of this unique representation of $z$ (i.e. $C\cong Ae\times A(1-e)$) we get 
$z_2=b+b'$  and $z_1=(a+a'+b+b')$. Now because $z$  and $(1-e)$ are in $\tilde \p$ we conclude that $z_1\in \tilde p$. But since 
$z_1\in A$, we get $z_1\in \tilde\p \cap A=\supp_A(\alpha)$. Now we know that $a,a',b,b'$ are in $\alpha$ and 
$a+a'+b+b'\in \supp_A(\alpha)$ which implies that $a,a',b,b'\in \supp_A(\alpha)\subset \tilde \p$. This shows specifically that
$ae+b\in \tilde \p$.

Now, considering $c(1-e)+de$. If $d\in \alpha$ then 
$$de + c(1-e)\in \alpha[e]+\tilde\p=\tilde \alpha$$
Else if $d\not\in \alpha$, because $A=\alpha\cup-\alpha$ we know that $d\in -\alpha$. So 
$$de + c(1-e)\in -\alpha[e]+\tilde\p=-(\alpha[e]+\tilde\p)=\tilde \alpha$$

\vspace{5mm}\noindent\emph{\textbf{Case 1.2:}} There is an $n\in \N$ such that $C=A[e_1,\dots, e_n]$ for some 
$e_1,\dots,e_n\in E(B)$.

\noindent We can prove this by induction. We have proven the case for $n=1$, so we may as well assume that $n>1$ and 
set $D=A[e_1,\dots,e_{n-1}]$ and assume that in case $C=D$ the Claim holds. In case $C=A[e_1,\dots,e_n]$, we have
$C=D[e_n]$. Assume that $\tilde \p\in \spec C$ and $\alpha\in \sper A$ such that $\tilde\p\cap A=\supp_A(\alpha)$. Write 
$\p'=\tilde \p \cap D$, then we know that $\p'\cap A=\supp_A(\alpha)$ as well. And so by our induction hypothesis there is an
$\alpha'\in \sper D$ such that $\supp_D(\alpha')=\p'=\p\cap D$ and that $\alpha'\cap A=\alpha$. Now using Case 1, we know that 
$\tilde \alpha$ defined by 
$$\tilde \alpha :=\alpha'[e_n]+\tilde \p$$ 
is in $\sper C$ and that $\supp_C(\tilde \alpha)=\tilde \p$ and also $\tilde \alpha \cap D = \alpha'$.
But this only implies that $\tilde \alpha \cap A= \alpha $. Thus we have proven the claim for this case.

\vspace{5mm}\noindent\emph{\textbf{Case 1.3:}}  $C=A[e|e\in E(B)]$  

\noindent This is the case that needs to be proven in general, but we shall make use of the other cases in order to prove this.
Define
$$\mathcal D:=\{A[e_1,\dots, e_n] \sep n\in \N, e_1,\dots,e_n \in E(B)\} $$
then we know that $\mathcal D$ can be considered as a directed set with 
$$\sup(A[e_{1},\dots,e_{n}],A[f_{1},\dots,f_{m}])=A[e_{1},\dots,e_{n},f_{1},\dots,f_{m}]\quad e_1,\dots,e_n,f_1,\dots,f_m \in E(B)$$
Moreover we also know that $\displaystyle\lim_{\stackrel{\longrightarrow}{D\in \mathcal D}} D = C$. 
By \cite{CR} Proposition 2.4, we also know that (in the category of topological spaces)
\begin{equation}\label{sperlim}
\sper C = \sper \lim_{\stackrel{\longrightarrow}{D\in \mathcal D}} D = \lim_{\stackrel{\longleftarrow}{D\in \mathcal D}} \sper D
\end{equation}

Let now $\tilde \p\in \spec C$ and $\alpha\in\sper A$ such that $\p\cap A=\supp(\alpha)$. For each $D\in \mathcal D$ we 
have a canonical injection $D\hookrightarrow C$. So, for each such $D\in \mathcal D$ let us consider 
$\p_D\in \spec D $ defined by $\p_D:=\tilde\p\cap D$. By Case 1.2 we know that for each $D\in\mathcal D$ 
there is an $\alpha_D\in \sper D$ such that $\supp_D(\alpha_D)=\p_D$ and $\alpha_D\cap A=\alpha$. We shall
in particular make use of the $\alpha_D$ as constructed in the proof of Case 1.2. More concretely, if $D=A[e_1,\dots,e_n]$
then $\alpha_D$ is defined by 
$$\alpha_D = \alpha[e_1,\dots,e_n] + \p_D$$

\noindent First we show that
if $D_1,D_2\in \mathcal D$ with $D_1\subset D_2$ then $\alpha_{D_1}\cap D_2 = \alpha_{D_2}$. We may write 
$$D_1 = A[e_1,\dots, e_n], D_2=A[e_2,\dots,e_m] \textrm{ for some } m>n, e_1,\dots,e_m \in E(B)$$
Obviously $\alpha_{D_2}\cap D_1 \supset \alpha_{D_1}$ and 
$$\supp_{D_1}(\alpha_{D_2}\cap D_1) = \p_{D_2}\cap D_1 = \p_{D_1} =\supp_{D_1}(\alpha_{D_1})$$
and by the \emph{basic property of the real spectrum} these all imply 
that $\alpha_{D_2}\cap D_1 = \alpha_{D_1}$ (because the real spectrum has the property that if a prime cone 
is in the closure of a prime cone, i.e. it specializes the other prime cone, then their image under $\supp$ is unequal).

Now by Equation \ref{sperlim} one can consider $\sper C$ as a subspace of $\prod_{D\in\mathcal D} \sper D$. And using the 
above analysis one can easily see that 
$$\{\alpha_D\}_{D\in\mathcal D} \in \sper C$$
We thus define $\tilde\alpha := \{\alpha_D\}_{D\in\mathcal D}$ and we easily see that $\tilde \alpha \cap A=\alpha$. Now we 
also note the fact that, in the category of topological spaces we have (see for instance \cite{EGAIV3} Corollaire 8.2.10)
\begin{equation}\label{speclim}
\spec C = \spec \lim_{\stackrel{\longrightarrow}{D\in \mathcal D}} D = \lim_{\stackrel{\longleftarrow}{D\in \mathcal D}} \spec D
\end{equation}
we then easily see that 
$$\{\p_D\}_{D\in\mathcal D} \in \spec C$$
and that one actually has $\tilde \p =\{\p_D\}_{D\in\mathcal D}$ and therefore $\supp_C(\tilde \alpha)=\tilde\p$. 

Note also that by the construction of projective limits in topological spaces and direct limits the category of porings, 
it is easy to see that $\tilde \alpha$ is no other than 
$$\tilde \alpha := \alpha[e\sep e\in E(B)] + \tilde \p$$
\blacksq

\vspace{5mm}\noindent The above Claim just showed that there is a continous surjection (continuity, due to the universal
property of fiber product in the category of topological space)
$$\phi: \sper C \lthrarrow \spec C \times_{\spec A} \sper A$$

\vspace{5mm}\noindent\underline{Claim 2:} $\phi$ is injective. 

\noindent\emph{Proof of Claim 2.} 
Let $\alpha\in \sper A$, $\tilde \p\in \spec C$ and $\tilde \alpha,\tilde \beta \in \sper C$ such that 
$$\supp(\tilde \alpha)=\supp(\tilde \beta)=\tilde \p \quad\textrm{ and }\quad \tilde \alpha\cap A=\tilde \beta\cap A = \alpha$$
then
$$ \tilde \alpha ,\tilde \beta \supset \alpha[e\sep e\in E(B)] + \tilde \p$$ 
and that by Claim 1 we also know that $\alpha[e\sep e\in E(B)] + \tilde \p$ is a prime cone of $C$ with image under $\supp_C$
being $\tilde \p$. But 
by the basic property of real spectra this proves us that 
$$\tilde \alpha = \tilde \beta = \alpha[e\sep e\in E(B)] + \tilde \p$$
and so $\phi$ is indeed injective. 
\blacksq

\vspace{5mm}\noindent\underline{Claim 3:} $\phi$ is open and thus a homeomorphism.

\noindent\emph{Proof of Claim 3.} It suffices to prove that for any 
$n\in \N$ and $b_1,\dots, b_n\in C$, one has that $\phi(P_C(b_1,\dots,b_n))$ is open (i.e. the image of any basic
open set is open). 
Again we work with different cases $\dots$ 

\vspace{5mm}\noindent\emph{\textbf{Case 3.1:}} $C=A[e]$ for some $e\in E(B)$

\noindent Define 
$$X:=\phi(\sper C)=\spec C\times_{\spec A}\sper A$$ 
then we claim that for any $n\in \N$ and $b_1,\dots,b_n\in C$ one has 
$$\phi(P_C(b_1,\dots,b_n)) =X\cap (D_C(e)\times P_A(b_{1,1},\dots,b_{1,n})\cup D_C(1-e)\times P_A(b_{2,1},\dots,b_{2,n}))$$
where $b_i=b_{1,i}e +b_{2,i}(1-e)$ (as explained in Case 1.1, this is the unique representation of $b_i\in C\cong Ae\times A(1-e)$).

"$\subset$" Let $\tilde \alpha\in P_C(b_1,\dots,b_n)$ with $\tilde p\in \spec C, \alpha\in\sper A$ such that 
$\phi(\tilde \alpha)=(\tilde \p,\alpha)$. Consider the case where $e\not\in\tilde\p$. If for an $i\in\{1,\dots,n\}$ one has 
$b_{1,i} \in -\alpha$, then $b_{1,i}e \in -\tilde \alpha$ (because $e\in\tilde\alpha$ for any $e\in E(C)$ and 
$\tilde\alpha\in\sper C$). 
But we know $b_i\in \tilde \alpha\backslash\supp_C(\tilde\alpha)$ with $1-e\in \supp_C(\tilde \alpha)$, so 
$$b_{1,i}e\in\tilde\alpha\Rightarrow b_{1,i}e\in\supp_C(\tilde \alpha)=\tilde \p\Rightarrow b_i=b_{1,i}e +b_{2,i}(1-e)\in \supp_C(\tilde\alpha)$$
And this is a contradiction! Thus for $e\not\in \tilde \p$ one has that 
$b_{1,i}\in \alpha\backslash\supp_A(\alpha)$ for all $i\in \{1,\dots,n\}$ (since $A=\alpha\cup -\alpha$), or in other 
words for $e\not\in\tilde \p$ we get 
$$(\tilde \p,\alpha)\in D_C(e)\times P_A(b_{1,1},\dots,b_{1,n})$$
Similarly one proves that for the case that $e\in \tilde \p$ one gets 
$$(\tilde \p,\alpha)\in D_C(1-e)\times P_A(b_{2,1},\dots,b_{2,n})$$

"$\supset$" Let 
$$(\tilde \p,\alpha)\in X\cap(D_C(e)\times P_A(b_{1,1},\dots,b_{1,n}))$$
we know by Claim 2 and Claim 1 that there exists a unique $\tilde \alpha \in \sper C$ such that 
$\phi(\tilde\alpha) = (\tilde \p,\alpha)$. Since $e\not \in \tilde p$ one has that $1-e\in \tilde p$ so 
for all $i\in \{1,\dots,n\}$ one gets 
$$b_i=b_{1,i}e + b_{2,i}(1-e) \in \tilde \alpha$$ 
because $\tilde\p=\supp_C(\tilde \alpha)$ and because 
$b_{1,i}\in \alpha\subset\tilde\alpha$ 
and so
$b_{1,i}e\in\tilde\alpha$. 

Now if $b_i\in \supp_C(\tilde\alpha)=\tilde\p$ we have 
$$b_{1,i}e \in \tilde \p\Rightarrow b_{1,i}\in \tilde\p\Rightarrow b_{1,i}\in \tilde\p\cap A=\supp_A(\alpha)$$
but this is a contradiction since we know from begining that 
$\alpha\in P_A(b_{1,1},\dots,b_{1,n})$ which implies that $b_{1,i}\in \alpha \backslash \supp_A(\alpha)$. Thus in general
we have that for any $i\in \{1,\dots,n\}$ 
$$b_i\in \tilde \alpha\backslash \supp_C(\tilde\alpha)$$
And so we can conclude that $\tilde\alpha\in P_C(b_1,\dots,b_n)$.

Similary one proves for the case
$$(\tilde \p,\alpha)\in X\cap(D_C(1-e)\times P_A(b_{2,1},\dots,b_{2,n}))$$
one gets $\tilde \alpha \in P_C(b_1,\dots,b_n)$, where $\tilde \alpha$ is the unique element in $\sper C$ such that 
$\phi(\tilde \alpha)= (\tilde \p,\alpha)$.

\vspace{5mm}\noindent
Thus for this Case we have
proven that $\phi$ is open (and thus a homeomorphism).

\vspace{5mm}\noindent\emph{\textbf{Case 3.2:}} There is an $m\in \N$ such that $C=A[e_1,\dots, e_m]$ for some 
$e_1,\dots,e_m\in E(B)$.

We prove this by induction over $m$. We know that this is true for the case $m=1$ (Case 3.1), thus we assume $m\geq 2$. 
For simplicity, define $C_{m-1} := A[e_1,\dots, e_{m-1}]$. By Case 3.1 and induction hypothesis 
we know that the commutative diagrams 
$$
\begindc{\commdiag}[1]
\obj(0,50)[A]{$\sper C_{m-1}$}
\obj(100,50)[B]{$\spec C_{m-1}$}
\obj(0,0)[C]{$\sper A$}
\obj(100,0)[D]{$\spec A$}
\mor{A}{B}{}
\mor{B}{D}{}
\mor{A}{C}{}
\mor{C}{D}{}
\enddc
$$
and 
$$
\begindc{\commdiag}[1]
\obj(0,50)[A]{$\sper C$}
\obj(100,50)[B]{$\spec C$}
\obj(0,0)[C]{$\sper C_{m-1}$}
\obj(100,0)[D]{$\spec C_{m-1}$}
\mor{A}{B}{}
\mor{B}{D}{}
\mor{A}{C}{}
\mor{C}{D}{}
\enddc
$$
are pullbacks in the category of topological spaces. It is a known fact in category theory that "pasting" the two 
commutative rectangle will give us the following commutative diagram
$$
\begindc{\commdiag}[1]
\obj(0,100)[A]{$\sper C$}
\obj(100,100)[B]{$\spec C$}
\obj(0,50)[C]{$\sper C_{m-1}$}
\obj(100,50)[D]{$\spec C_{m-1}$}
\obj(0,0)[E]{$\sper A$}
\obj(100,0)[F]{$\spec A$}
\mor{A}{B}{}
\mor{B}{D}{}
\mor{A}{C}{}
\mor{C}{D}{}
\mor{C}{E}{}
\mor{E}{F}{}
\mor{D}{F}{}
\enddc
$$
whose outer rectangle is also a pullback in the category of topological spaces (see for instance \cite{AHS} 
Proposition 11.10). And this proves that $\phi$ is also a homeomorphism for this case.

\vspace{5mm}\noindent\emph{\textbf{Case 3.3:}}  $C=A[e|e\in E(B)]$  

\noindent This is the general case and the openness of $\phi$ 
is easily seen by the general definition of limit topology and noting Equations
(\ref{sperlim}) and (\ref{speclim}) in Case 1.3.
\blacksq
With this we have also proven the Theorem.
\end{proof}

\begin{ack}
I would like to thank Oliver Delzeith for his most valuable insights and mathematical discussions with me. In times of need
and confusion I have constantly relied on him. 
\end{ack}

\end{document}